\documentclass[11pt,a4paper,openany]{amsart}

\usepackage{amsmath,amsfonts,ams}
\usepackage{latexsym}
\usepackage{amssymb}
\usepackage{color}
\usepackage{epsfig}
\usepackage{graphicx}
\usepackage{subfigure}
\usepackage{mathrsfs}
\usepackage{amscd}
\usepackage{amsthm}
\usepackage{color}
\usepackage{cite}
\usepackage{epstopdf}

\numberwithin{equation}{section}
\newtheorem{proposition}{Proposition}[section]
\newtheorem{definition}{Definition}[section]
\newtheorem{lemma}{Lemma}[section]
\newtheorem{theorem}{Theorem}[section]
\newtheorem{corollary}{Corollary}[section]
\newtheorem{remark}{Remark}[section]

\let\pa=\partial

\def\C{\mathop{\bf C\kern 0pt}\nolimits}
\def\DD{\mathop{\bf D\kern 0pt}\nolimits}
\def\K{\mathop{\bf K\kern 0pt}\nolimits}
\def\N{\mathop{\bf N\kern 0pt}\nolimits}
\def\Q{\mathop{\bf Q\kern 0pt}\nolimits}

\newcommand{\la}{\lambda}

\newcommand{\beq}{\begin{equation}}
\newcommand{\eeq}{\end{equation}}
\newcommand{\ben}{\begin{eqnarray}}
\newcommand{\een}{\end{eqnarray}}
\newcommand{\beno}{\begin{eqnarray*}}
\newcommand{\eeno}{\end{eqnarray*}}

\makeatletter
\newcommand{\Extend}[5]{\ext@arrow0099{\arrowfill@#1#2#3}{#4}{#5}}
\makeatother

\begin{document}
\title{ Scattering Theory for   Energy-Supercritical   Klein-Gordon  Equation}

\author{Changxing Miao}
\address{Institute of Applied Physics and Computational Mathematics,
P. O. Box 8009,\ Beijing,\ China,\ 100088;} \email{miao\_changxing@iapcm.ac.cn}

\author{Jiqiang Zheng}
\address{The Graduate School of China Academy of Engineering Physics, P. O. Box 2101, Beijing, China, 100088}
\email{zhengjiqiang@gmail.com} \maketitle
\begin{abstract}
In this paper, we consider the question of the global well-posedness
and scattering for the cubic Klein-Gordon equation $u_{tt}-\Delta
u+u+|u|^2u=0$ in dimension $d\geq5$. We show that if the solution
$u$ is apriorily bounded in the critical Sobolev space, that is,
$(u, u_t)\in L_t^\infty(I; H^{s_c}_x(\R^d)\times H_x^{s_c-1}(\R^d))$
with $s_c:=\frac{d}2-1>1$, then $u$ is global and scatters. The
impetus to consider this problem stems from a series of recent works
for the energy-supercritical nonlinear wave equation and nonlinear
Schr\"odinger equation. However, the scaling invariance is broken in
the Klein-Gordon equation.  We will utilize the concentration
compactness ideas to show that the proof of the global
well-posedness and scattering is reduced to disprove the existence
of the scenario: soliton-like solutions.  And such solutions are
precluded by making use of the Morawetz inequality, finite speed of
propagation and concentration of potential energy.
\end{abstract}

\begin{center}
 \begin{minipage}{120mm}
   { \small {\bf Key Words: Klein-Gordon equation;  scattering theory; Strichartz estimate; Energy supercritical; concentration compactness}
      {}
   }\\
    { \small {\bf AMS Classification:}
      {35P25, 35B40, 35Q40.}
      }
 \end{minipage}
 \end{center}
\section{Introduction}
\setcounter{section}{1}\setcounter{equation}{0} This paper is
devoted to the study of the Cauchy problem of the cubic Klein-Gordon
equation
\begin{align} \label{equ1}
\begin{cases}    \ddot{u} - \Delta u  +  u +  f(u)=  0,  \qquad  (t,x) \in
\mathbb{R}\times\mathbb{R}^d,\ d\geq5,\\
\big(u(0,x),~u_t(0,x)\big)=\big(u_0(x),~u_1(x)\big)\in
H^{s_c}(\R^d)\times H^{s_c-1}(\R^d),\end{cases}
\end{align}
where $f(u)=|u|^2u,~u$ is a real-valued function defined in
$\mathbb{R}^{1+d}$, the dot denotes the time derivative, $\Delta$ is
the Laplacian in $\mathbb{R}^{d},\ s_c:=\frac{d}2-1$.

Formally, the solution $u$ of \eqref{equ1} conserves  the energy
\begin{equation*}\label{econ}
\aligned E(u(t),\dot{u}(t))=&\frac12 \int_{\mathbb{R}^d}
\Big(\big|\dot{u}(t,x) \big|^2 +\big| \nabla u(t,x) \big|^2 + \big|
u(t,x) \big|^2 \Big)dx + \frac{1}{4} \int_{\mathbb{R}^d}
 |u|^4dx\\
\equiv&E(u_0,u_1).
\endaligned
\end{equation*}

The class of solutions to wave equation $$\ddot{u} - \Delta u +
|u|^2u=0 $$ is left invariant by the scaling
\begin{equation}\label{scale}
u(t,x)\mapsto \la\ u(\la\ t,\la\ x),\quad \forall~\la>0.
\end{equation}
Moreover, it leaves the Sobolev norm $\dot{H}^{s_c}_x(\R^d)$ with
$s_c=\frac{d}2-1$ invariant. Since
$s_c>1,$ it is called the energy-supercritical.

The scattering theory for the Klein-Gordon equation with
$f(u)=\mu|u|^{p-1}u$  has been intensively studied in
\cite{Br84,Br85,GiV85b,IMN,Na99b,Na01}. For $\mu=1$ and
\begin{equation}\label{}
   1+\frac{4}{d}<p<1+\gamma_d\frac{4}{d-2},
\quad
 \gamma_d=\left\{ \aligned
    &1 ,& 3\leq d\leq 9; \\
     &\frac{d}{d+1},& d\geq 10.
\endaligned
\right.
\end{equation}
Brenner \cite{Br85} established the scattering results in the energy
space $H_x^1(\R^d)\times L_x^2(\R^d)$, which does not contain all
subcritical cases for $d\geq 10$. Thereafter, Ginibre and Velo
\cite{GiV85b} exploited the Birman-Solomjak space $\ell^m(L^q,I,B)$
in \cite{BiS75} and the delicate estimates to improve the results in
\cite{Br85}, which covered all subcritical cases. Finally K.
Nakanishi \cite{Na99b} obtained the scattering results for the
critical case ($p=1+\frac4{d-2}$) by the strategy of induction on
energy \cite{CKSTT07} and a new Morawetz-type estimate. And
recently, S. Ibrahim, N. Masmoudi and K. Nakanishi \cite{IMN, IMN1}
utilized the concentration compactness ideas to give the scattering
threshold for the focusing ($\mu=-1$) nonlinear Klein-Gordon
equation. Their method also works for the defocusing case.

In this paper, we  consider the cubic Klein-Gordon equation in
dimension $d\geq5$, which is the energy-supercritical case. Such
results have been recently established for many other equations
including the nonlinear wave equation (NLW) and the nonlinear
Schr\"odinger equation (NLS), since Kenig-Merle \cite{KM2011} on NLW
for radial solutions in $\R^3$. We also refer to
\cite{Blut2012,Blut2011,Blut20115,DKM,KM2011D,KV2010,KV,KV2011}.

To be more precise, let us recall the results for the
energy-supercritical nonlinear wave equation
$$\ddot{u}-\Delta u+|u|^pu=0,~(t,x)\in\R\times\R^d,~p>\frac4{d-2}.$$
In dimension three, Kenig and Merle \cite{KM2011} proved that if the
radial solution $u$ is apriorily bounded in the critical Sobolev
space, that is, $(u,u_t)\in L_t^\infty(I;
\dot{H}^{s_c}_x(\R^d)\times \dot{H}^{s_c-1}_x(\R^d))$ with
$s_c:=\frac{d}2-\frac2p>1$, then $u$ is global and scatters. In
\cite{KM2011D}, they also considered the radial solutions in odd
dimensions. Later, Killip and Visan \cite{KV} showed the result in
$\R^3$ for the non-radial solutions by making use of Huygens
principal and so called ``localized double Duhamel trick". Further,
they \cite{KV2011} proved the radial solution in all dimensions in
some ranges of $p$. Thereafter, Bulut
\cite{Blut2012,Blut2011,Blut20115} proved the results in dimensions
$d\geq5$ for the cubic nonlinearity (i.e. $p=2$).
 Recently, Duyckaerts, Kenig and Merle \cite{DKM} obtain such result for the focusing wave equation with radial solution
in three dimension. Their proof relies on the compactness/rigidity
method, pointwise estimates on compact solutions obtained in
\cite{KM2011}, and channels of energy arguments used by the authors
in previous works \cite{DKM2011} on the energy-critical equation.

Before stating the main result, we introduce some background
materials.
\begin{definition}[solution]\label{def1.1}
 A function $u:~I\times\R^d\to\R$ on a nonempty time
interval $I$ containing zero is a strong solution to \eqref{equ1} if
$(u,u_t)\in C_t^0(J; H^{s_c}_x(\R^d)\times H^{s_c-1}_x(\R^d))$ and
$u\in [W](J)$ $($defined in \eqref{dwi}$)$ for any compact interval
$J\subset I$ and for each $t\in I$, it obeys the Duhamel formula:
\begin{equation}\label{duhamel}
{u(t)\choose \dot{u}(t)} = V_0(t){u_0(x) \choose u_1(x)}
-\int^{t}_{0}V_0(t-s){0 \choose f(u(s))} ds,
\end{equation}
where
$$  V_0(t) = {\dot{K}(t), K(t)
\choose \ddot{K}(t), \dot{K}(t)}, \quad
K(t)=\frac{\sin(t\omega)}{\omega},\quad
\omega=\big(1-\Delta\big)^{1/2}.$$ We refer to the interval $I$ as
the lifespan of $u$. We say that $u$ is a maximal-lifespan solution
if the solution cannot be extended to any strictly larger interval.
We say that $u$ is a global solution if $I=\R.$

\end{definition}

The solution lies in the space $ [W](I)$ locally in time is natural
since by Strichartz estimate, the linear flow always lies in this
space. Also,  if a solution $u$ to \eqref{equ1} is global, with
$\|u\|_{W(\R)}<+\infty$, then it scatters in both time directions in
the sense that
 there exist  solutions $v_\pm$ of the free Klein-Gordon equation
\begin{equation}\label{le}
    \ddot{v} - \Delta v  + v =  0
\end{equation}
with $(v_\pm(0), \dot{v}_\pm(0))\in H^{s_c}_x(\R^d)\times
H^{s_c-1}_x(\R^d)$ such that
\begin{equation}\label{1.2}
\Big\|\big(u(t), \dot{u}(t)\big)-\big(v_\pm(t),
\dot{v}_\pm(t)\big)\Big\|_{H^{s_c}_x\times H^{s_c-1}_x}
\longrightarrow 0,\quad \text{as}\quad t\longrightarrow
\pm\infty.\end{equation} In view of this, we define
\begin{equation}\label{scattersize}
 S_I(u)=\|u\|_{
[W](I)}
\end{equation}
as the scattering size of $u$, where
\begin{equation}\label{dwi}
[W](I)=L_t^{\frac{2(d+1)}{d-1}}\big(I;B^{\frac{d-3}2}_{\frac{2(d+1)}{d-1},2}(\mathbb{R}^d)\big).
\end{equation}
Closely associated with the notion of scattering is the notion of
blowup:
\begin{definition}[Blowup]\label{def1.2}  Let $u:I\times\R^d\to\mathbb{C}$ be a  maximal-lifespan solution to \eqref{equ1}. If there
exists a time $t_0\in I$ such that $S_{[t_0,\sup I)}(u)=+\infty$,
then we say that the solution $u$
 blows up forward in time. Similarly, if there exists a time $t_0\in
 I$ such that $S_{(\text{inf}~
I,t_0]}(u)=+\infty$, then we say that  $u(t,x)$ blows up backward in
time.
\end{definition}

Now we state our main result.

\begin{theorem}\label{theorem}
Assume that $d\geq5,$ and $s_c:=\frac{d}2-1$. Let
$u:~I\times\R^d\to\R$ be a maximal-lifespan solution to \eqref{equ1}
such that
\begin{equation}\label{assume1}
\big\|(u,u_t)\big\|_{L_t^\infty(I;H^{s_c}_x(\R^d)\times
H^{s_c-1}_x(\R^d))}<+\infty.
\end{equation} Then the solution $u$ is global and  scatters.
\end{theorem}

%\begin{remark} Using the same argument as this paper and the fraction
%chain rule as in \cite{KV2010}, one can  extend the above result to
%the more general nonlinear term $f(u)=|u|^pu$ with
%$p\geq\frac4{d-2}$ and $s_c=\frac{d}2-\frac2p\geq1$ if $p$ is an
%even integer or $1\leq s_c<1+p$ otherwise in dimension $d\geq3.$
%\end{remark}

\noindent{\bf The outline of the proof of Theorem \ref{theorem}:}
For any $0\leq E_0<+\infty,$ we  define
$$L(E_0):=\sup\Big\{S_I(u):~u:~I\times\R^d\to \R\ \text{such\ that}\
\sup_{t\in I}\big\|(u,u_t)\big\|_{H^{s_c}_x\times H^{s_c-1}_x}^2\leq
E_0\Big\},$$ where the supremum is taken over all solutions
$u:~I\times\R^d\to \R$ to \eqref{equ1} satisfying
$\big\|(u,u_t)\big\|_{H^{s_c}\times H^{s_c-1}}^2\leq E_0.$ Thus,
$L:\ [0,+\infty)\to [0,+\infty)$ is a non-decreasing function.
Moreover, from  the small data theory, see Theorem \ref{small}, we
know that
$$L(E_0)\lesssim E_0^\frac12\quad \text{for}\quad E_0\leq\eta_0^2,$$
where $\eta_0=\eta(d)$ is the threshold from the small data theory.

From the stability theory (see Theorem \ref{long}), we see that $L$ is
continuous. Therefore, there must exist a unique critical
$E_c\in(0,+\infty]$ such that $L(E_0)<+\infty$ for $E_0<E_c$ and
$L(E_0)=+\infty$ for $E_0\geq E_c$. In particular, if
$u:~I\times\R^d\to \R$ is a maximal-lifespan solution to
\eqref{equ1} such that $\sup\limits_{t\in
I}\big\|(u,u_t)\big\|_{H^{s_c}\times H^{s_c-1}}^2<E_c,$ then $u$ is
global and moreover,
$$S_\R(u)\leq L\big(\big\|(u,u_t)\big\|_{L_t^\infty(\R;H^{s_c}\times
H^{s_c-1})}^2\big).$$ The proof of Theorem \ref{theorem} is
equivalent to show $E_c=+\infty.$  We argue by contradiction. We
show that if $E_{c}<+\infty$, then there exists a nonlinear global
solution of \eqref{equ1} with $L_t^\infty(\R;H_x^{s_c}\times
H_x^{s_c-1})$-norm be exactly $E_{c}$. Moreover, this solution
satisfies some strong compactness properties. This is completed in
Section 4 where we utilize the profile decomposition that was
established in Ibrahim, Masmoudi and Nakanishi \cite{IMN}, and a
strategy introduced by Kenig and Merle \cite{KM}. Finally, we
utilize the finiteness of the energy to show that the solutions
obtained in Section 4 are not possible. More precisely, by Morawetz
inequality \cite{Br85, MC,MoS72}
\begin{equation}\label{moraw}
\int_0^T\int_{\R^d}\frac{|u(t,x)|^4}{|x|}dxdt\lesssim
E(u,u_t),\end{equation} we know that the left-hand side of
\eqref{moraw} is bounded by the energy for any $T>0$. On the other
hand, notice that by finite speed of propagation and concentration
of potential energy, the left-hand side of \eqref{moraw} should grow
logarithmical in time $T$. This gives a contradiction by choosing
$T$ sufficiently large.

The paper is organized as follows. In Section $2$, we deal with the
local theory
 for the equation \eqref{equ1}.  In Section $3$, we give
the linear and nonlinear profile decomposition and show some
properties of the profile. Thereafter, we extract a critical
solution in Section $4$. Finally in Section $5$, we preclude the
critical solution, which completes the proof of Theorem
\ref{theorem}.

We conclude the introduction by giving some notations which will be
used throughout this paper. We always assume the spatial dimension
$d\geq 5$ and $f(u)=|u|^2u$. For any $r, 1\leq r \leq \infty$, we
denote by $\|\cdot \|_{r}$ the norm in $L^{r}=L^{r}(\mathbb{R}^d)$
and by $r'$ the conjugate exponent defined by $\frac{1}{r} +
\frac{1}{r'}=1$. For any $s\in \mathbb{R}$, we denote by
$H^s(\mathbb{R}^d)$ the usual Sobolev space. Let $\psi\in
\mathcal{S}(\mathbb{R}^d)$ be such that $\text{supp}\
{\widehat{\psi}} \subseteq \big\{\xi: \frac{1}{2} \leq|\xi| \leq 2
\big\}$ and $ \sum_{j\in \mathbb{Z}} \widehat{\psi} (2^{-j} \xi) = 1
$ for $\xi \neq 0.$ Define $\psi_0$ by $\widehat{\psi}_0 = 1 -
 \sum_{j\geq 1} \widehat{\psi} (2^{-j} \xi).$  Thus $\text{supp}\
\widehat{\psi}_0 \subseteq \big\{\xi: |\xi| \leq 2 \big\}$ and
$\widehat{\psi}_0 = 1$ for $|\xi| \leq 1$. We denote by $\Delta_j$
and $\mathcal{P}_0$ the convolution operators whose symbols are
respectively given by $\widehat{\psi}(\xi/2^{j})$ and
$\widehat{\psi}_0(\xi)$. For $s \in \mathbb{R}, 1\leq r \leq
\infty$, the  Besov spaces $ B^{s}_{r, 2}(\mathbb{R}^d)$ and $\dot
B^{s}_{r,2}(\R^d)$ are defined by \begin{align*} B^{s}_{r,
2}(\mathbb{R}^d) =& \bigg\{ u \in \mathcal{S}'(\mathbb{R}^d),
\|\mathcal{P}_0 u\|^2_{L^r}+ \big\|2^{js} \|\Delta_j u\|_{L^r}
\big\|^2_{l^2_{j\in \mathbb{N}}} < \infty \bigg\}\\
\dot B^{s}_{r, 2}(\mathbb{R}^d) =& \bigg\{ u \in
\mathcal{S}'_h(\mathbb{R}^d), \big\|2^{js} \|\Delta_j u\|_{L^r}
\big\|^2_{l^2_{j\in \mathbb{Z}}} < \infty \bigg\}.
\end{align*} For
details of Besov space, we refer to \cite{BL76}. For any interval
$I\subset \mathbb{R}$ and any Banach space $X$ we denote by ${\mathcal
C}(I; X)$ the space of strongly continuous functions from $I$ to $X$
and by $L^q(I; X)$ the space of strongly measurable functions from
$I$ to $X$ with $\|u(\cdot); X\|\in L^q(I).$  We denote by
$\langle\cdot, \cdot\rangle$ the scalar product in $L^2$.

%%%%%%%%%%%%%%%%%%%%%%%%%%%%%%%%%%%%%%%%%%%%%%%%%%%%%%%%%%%%%%%%%%%
\section{Preliminaries}
 \setcounter{section}{2}\setcounter{equation}{0}

 \subsection{Strichartz estimate and local theory}
In this section, we consider the Cauchy problem  for the equation
$(\ref{equ1})$
\begin{equation} \label{equ2}
    \left\{ \aligned &\ddot{u} - \Delta u  +  u +  f(u)=  0, \\
    &u(0)=u_0,~\dot{u}(0)=u_1.
    \endaligned
    \right.
\end{equation}
The integral equation for the Cauchy problem $(\ref{equ2})$ can be
written as
\begin{equation}\label{inte1}
u(t)=\dot{K}(t)u_0 + K(t)u_1-\int^{t}_{0}K(t-s)f(u(s))ds,
\end{equation}
or
\begin{equation}\label{inte2}
{u(t)\choose \dot{u}(t)} = V_0(t){u_0(x) \choose u_1(x)}
-\int^{t}_{0}V_0(t-s){0 \choose f(u(s))} ds,
\end{equation}
where
$$K(t)=\frac{\sin(t\omega)}{\omega}, \quad V_0(t) = {\dot{K}(t), K(t)
\choose \ddot{K}(t), \dot{K}(t)}, \quad \omega=\big( 1-\Delta\big)^{1/2}.$$

Let $U(t)=e^{it\omega}$, then
\begin{equation*}
\dot{K}(t)= \frac{U(t)+U(-t)}{2}, \qquad  K(t)=
\frac{U(t)-U(-t)}{2i\omega}.
\end{equation*}

Now we recall the following dispersive estimate for the operator
$U(t)=e^{it\omega}$.
\begin{lemma}[\cite{Br85,GiV85b}]\label{lem21}
Let $2\leq r\leq \infty$ and $0\leq \theta\leq 1$. Then
\begin{equation*}
\big\|e^{i\omega t}f
\big\|_{B^{-(d+1+\theta)(\frac12-\frac1r)/2}_{r, 2}} \leq \mu(t)
\big\|f\big\|_{B^{(d+1+\theta)(\frac12-\frac1r)/2}_{r', 2}},
\end{equation*}
where
\begin{equation*}
\mu(t)=C \min\bigg\{ |t|^{-(d-1-\theta)(\frac12-\frac1r)_{+} },
|t|^{-(d-1+\theta)(\frac12-\frac1r)}\bigg\}.
\end{equation*}
\end{lemma}

According to the above lemma, the abstract duality and interpolation
argument(see \cite{GiV95}, \cite{KeT98}), we have the following
Strichartz estimates.
\begin{lemma}[\cite{Br85,GiV85b,MC,MZF}]\label{lem22}
Let $0\leq\theta_i \leq 1$, $\rho_i \in \mathbb{R}$, $2\leq q_i, r_i
\leq +\infty,~i=1,2$. Assume that $(\theta_i,d,q_i,r_i)\neq
(0,3,2,+\infty)$ satisfy the following admissible conditions
\begin{equation}
\left\{ \aligned \label{rl} 0\leq \frac{2}{q_i} &\leq
\min\Big\{(d-1+\theta_i)\Big(\frac{1}{2}-\frac{1}{r_i}\Big),
1\Big\},~~~i=1,2
   \\
&\rho_1+(d+\theta_1)\Big(\frac{1}{2}-\frac{1}{r_1}\Big)-\frac{1}{q_1}
=\mu,
\\&\rho_2+(d+\theta_2)\Big(\frac{1}{2}-\frac{1}{r_2}\Big)-\frac{1}{q_2}
=1-\mu.
\endaligned\right.
\end{equation}
Then, for $g \in H^\mu_x(\R^d)$, we have
\begin{align}\label{str1}
\big\| U(\cdot) g\big\|_{L^{q_1}\big(\mathbb{R}; B^{\rho_1}_{r_1, 2}
\big)} &\leq C \|g\|_{H^\mu};\\\label{str2} \big\| K_{R}\ast f
\big\|_{L^{q_1}\big(I; B^{\rho_1}_{r_1,  2} \big)} &\leq C\big\|  f
\big\|_{L^{q_2'}\big(I; B^{-\rho_2}_{r'_2, 2} \big)}.
\end{align}
where the subscript $R$ stands for retarded, and
\begin{align*}
K_R\ast f&=\int_{0}^tK(t-s)f(u(s))ds.
\end{align*}
\end{lemma}

Now it is useful to define several spaces and give estimates of the
nonlinearities in terms of these spaces. Define
$$ST(I)=[W](I),$$ where
$$[W](I)=L_t^{\frac{2(d+1)}{d-1}}\big(I;B^{\frac{d-3}2}_{\frac{2(d+1)}{d-1},2}(\mathbb{R}^d)\big).$$
In addition to the $ST$-norm, we also need the corresponding dual
norm
$$[W]^*(I)=L_t^{\frac{2(d+1)}{d+3}}\big(I;B^{\frac{d-3}2}_{\frac{2(d+1)}{d+3},2}(\mathbb{R}^d)\big).$$
Then we have by Strichartz estimate
\begin{align}\nonumber
&\big\|u\big\|_{[W](I)}+\big\|(u,u_t)\big\|_{L_t^\infty(I;
H^{s_c}_x\times H^{s_c-1}_x)}\\\label{str} \leq&
C\big\|(u_0,u_1)\big\|_{H^{s_c}_x\times H^{s_c-1}_x}+
C\big\|f(u)\big\|_{[W]^\ast(I)\oplus L_t^1(I;
B^{s_c-1}_{2,1}(\R^d))},
\end{align}
where the time interval $I$ contains zero.

\begin{lemma}[Product rule \cite{CW,Taylor}]\label{moser}
Let $s\geq0$, and $1<r,p_j,q_j<+\infty$ be such that
$\frac1r=\frac1{p_i}+\frac1{q_i}~(i=1,2).$ Then, we have
$$\big\||\nabla|^s(fg)\big\|_{L_x^r(\R^d)}\lesssim\|f\|_{{L_x^{p_1}(\R^d)}}\big\||\nabla|^sg
\big\|_{{L_x^{q_1}(\R^d)}}+\big\||\nabla|^sf\big\|_{{L_x^{p_2}(\R^d)}}\|g\|_{{L_x^{q_2}(\R^d)}}.$$
\end{lemma}

As a direct consequence, we have the following nonlinear estimate.

\begin{lemma}[Nonlinear estimate]\label{nlest}
Let $I$ be a time slab, one has
\begin{align}\nonumber
&\big\|u^2v\big\|_{[W]^\ast(I)}\\\label{nle}
\lesssim&\big\|u\big\|_{[W](I)}^{1+\frac{2}{d-1}}\big\|u\big\|_{L_t^\infty\dot{H}^{s_c}_x}^\frac{d-3}{d-1}
\big\|v\big\|_{[W](I)}^{\frac{2}{d-1}}\big\|v\big\|_{L_t^\infty\dot{H}^{s_c}_x}^\frac{d-3}{d-1}+
\big\|v\big\|_{[W](I)}\big\|u\big\|_{[W](I)}^{\frac{4}{d-1}}\big\|u\big\|_{L_t^\infty\dot{H}^{s_c}_x}^\frac{2(d-3)}{d-1}.
\end{align}
\end{lemma}

\vskip 0.2cm
\begin{proof}
It follows from the above product rule and Sobolev
embedding: $B^s_{p,\min\{p,2\}}(\R^d)\subset
F^s_{p,2}(\R^d)=W^{s,p}_x(\R^d)\subset B^s_{p,\max\{p,2\}}(\R^d)$
that
\begin{align}\label{nle1}
\big\|u^2v\big\|_{[W]^\ast(I)}
\lesssim\big\|u\big\|_{[W](I)}\big\|u\big\|_{L_{t,x}^{d+1}}\big\|v\big\|_{L_{t,x}^{d+1}}+\big\|v\big\|_{[W](I)}\big\|u\big\|_{L_{t,x}^{d+1}}^2.
\end{align}

Using H\"older's inequality and Sobolev embedding, we obtain
\begin{align*}
\big\|u\big\|_{L_{t,x}^{d+1}}\lesssim&\big\|u\big\|_{L_t^\frac{2(d+1)}{d-1}L_x^\frac{2d(d+1)}{d+3}}^\frac{2}{d-1}\big\|u\big\|_{L_t^\infty
L_x^d}^\frac{d-3}{d-1}\\
\lesssim&\big\|u\big\|_{[W](I)}^{\frac{2}{d-1}}\big\|u\big\|_{L_t^\infty\dot{H}^{s_c}}^\frac{d-3}{d-1}.
\end{align*}
Plugging this into \eqref{nle1}, we get \eqref{nle}.
\end{proof}

We can now state the local well-posedness for $(\ref{equ1})$ with
large initial data and small data scattering in the space
$H^{s_c}(\R^d)\times H^{s_c-1}(\R^d)$, which is the first step to
obtain the global time-space estimate and then
 lead to the scattering.

\begin{theorem}[Local wellposedness]\label{small}
Assume $(u_0,u_1)\in H^{s_c}_x(\mathbb{R}^d)\times
H^{s_c-1}_x(\mathbb{R}^d)$. There exists a small constant
$\delta=\delta(E)$ such that if $\|(u_0,u_1)\|_{ H^{s_c}\times
H^{s_c-1}}\leq E$ and $I$ is an time interval containing zero such
that
\begin{equation}\label{jiashe}
\big\|\dot{K}(t)u_0 + K(t)u_1\big\|_{[W](I)}\leq
\delta,\end{equation} then there exists a unique strong solution $u$
to \eqref{equ1} in $I\times \mathbb{R}^d$, with $u\in
C(I;H^{s_c}_x(\R^d))\cap C^1(I;H^{s_c-1}_x(\R^d))$ and
\begin{equation}\label{smalll}
\|u\|_{[W](I)}\leq 2\delta,~\|(u,u_t)\|_{L_t^\infty(I; H^{s_c}\times
H^{s_c-1})}\leq 2CE,
\end{equation}
where $C$ is the Strichartz constant as in Lemma \ref{lem22}.

In particular, if  $\|(u_0,u_1)\|_{ H^{s_c}\times
H^{s_c-1}}\leq\delta$, then the solution $u$ is global and scatters.
\end{theorem}
\begin{proof}
We apply the Banach fixed point argument to prove this lemma. First
we define the solution map
\begin{equation}\label{inte3}
\Phi(u(t))=\dot{K}(t)u_0 + K(t)u_1-\int^{t}_{0}K(t-s)f(u(s))ds
\end{equation}
on the complete metric space $B$
\begin{align*}
B=\big\{u\in C(I;H^{s_c}): \big\|(u,u_t)\big\|_{L_t^\infty(I;
H^{s_c}\times H^{s_c-1})}\leq 2CE, \|u\|_{[W](I)}\leq2\delta\big\}
\end{align*}
with the metric $d(u,v)=\big\|u-v\big\|_{[W](I)\cap L_t^\infty
H^{s_c}_x}$.

It suffices to prove that the operator defined by the RHS of
$(\ref{inte3})$ is a contraction map on $B$ for $I$. If $u\in B,$
then by Strichartz estimate \eqref{str}, \eqref{nle} and
\eqref{jiashe}, we have
\begin{align*}
\big\|\Phi(u)\big\|_{[W](I)}\leq&\big\|\dot{K}(t)u_0 +
K(t)u_1\big\|_{[W](I)}+C\big\|u^3\big\|_{[W]^\ast(I)}\\
\leq&\delta+C\big\|u\big\|_{[W](I)}^{1+\frac{4}{d-1}}\big\|u\big\|_{L_t^\infty\dot{H}^{s_c}}^\frac{2(d-3)}{d-1}.
\end{align*}
Plugging the assumption $\|u\|_{L^\infty(I;H^{s_c})}\leq 2CE$ and
$\|u\|_{[W](I)}\leq2\delta$, we see that for $u\in B$,
\begin{align*}
\big\|\Phi(u)\big\|_{[W](I)}
\leq&\delta+C(2\delta)^{1+\frac{4}{d-1}}(2CE)^\frac{2(d-3)}{d-1}.
\end{align*}
Thus we can choose $\delta$ small depending on $E$ and the
Strichartz constant $C$ such that
$$\big\|\Phi(u)\big\|_{[W](I)}\leq2\delta.$$
Similarly, if $u\in B,$ then
$\big\|(\Phi(u),\pa_t\Phi(u))\big\|_{L_t^\infty(I; H^{s_c}\times
H^{s_c-1})}\leq 2CE.$ Hence $\Phi(u)\in B$ for $u\in B.$ That is,
the functional $\Phi$
 maps the set $B$ back to itself.

On the other hand, by a same argument as before and Lemma
\ref{nlest}, we have for $u, v\in B$,
\begin{align*}
d(\Phi(u),\Phi(v))\lesssim&C\|u^3-v^3\|_{[W]^\ast(I)}\\
\leq&16C\|u-v\|_{[W](I)\cap L_t^\infty H^
{s_c}}\|(u,v)\|_{[W](I)}^{\frac2{d-1}}\big\|(u,v)\big\|_{[W](I)\cap
L_t^\infty H^ {s_c}}^\frac{2(d-2)}{d-1}\\
\leq&16C(4\delta)^\frac2{d-1}(4CE+2\delta)^\frac{2(d-2)}{d-1}d(u,v)
\end{align*}
which allows us to derive
\begin{equation*}
d(\Phi(u),\Phi(v))\leq\frac{1}{2}d(u,v),
\end{equation*}
by taking $\delta$ small such that
$$16C(4\delta)^\frac2{d-1}(4CE+2\delta)^\frac{2(d-2)}{d-1}\leq\frac12.$$

A standard fixed point argument gives a unique solution $u$ of
\eqref{equ1} on $I\times\R^d$ which satisfies the bound
\eqref{smalll}.
\end{proof}
Using Theorem \ref{small} as well as its proof, one easily derives the
following local theory for \eqref{equ1}. We omit the standard detail
here.
\begin{theorem}\label{lwp}
Assume that $d\geq5$, $s_c=\frac{d}2-1.$ Then, given  $(u_0,u_1)\in
H^{s_c}(\R^d)\times H^{s_c-1}(\R^d)$ and $t_0\in\R$, there exists a
unique maximal-lifespan solution $u: I\times\R^d\to\R$ to
\eqref{equ1} with initial data
$\big(u(t_0),u_t(t_0)\big)=\big(u_0,u_1\big)$. This solution also
has the following properties:
\begin{enumerate}
\item (Local existence) $I$ is an open neighborhood of $t_0$.
\item (Blowup criterion) If\ $\sup (I)$ is finite, then $u$ blows up
forward in time (in the sense of Definition \ref{def1.2}). If $\inf
(I)$ is finite, then $u$ blows up backward in time.
\item (Scattering) If $\sup (I)=+\infty$ and $u$ does not blow up
forward in time, then $u$ scatters forward in time in the sense
\eqref{1.2}. Conversely, given $(v_+,\dot{v}_+)\in
H^{s_c}(\R^d)\times H^{s_c-1}(\R^d)$ there is a unique solution to
\eqref{equ1} in a neighborhood of infinity so that \eqref{1.2}
holds.
\end{enumerate}
\end{theorem}

\subsection{Perturbation lemma}
In this part, we give the perturbation theory of the solution of
\eqref{equ1} with the global space-time estimate.

With any real-valued function $u(t,x)$, we associate the
complex-valued function $\vec{u}(t,x)$ by
\begin{equation}
\vec{u}=\langle\nabla\rangle^{s_c-1}\big(\langle\nabla\rangle
u-i\dot{u}\big),\quad u=\Re\langle\nabla\rangle^{-s_c}\vec{u},
\end{equation}
where $\Re z$ denotes the real part of $z\in\mathbb{C}$. Then the
free and nonlinear Klein-Gordon equations are given by
\begin{align}
\begin{cases}
(\Box+1)u=0\Longleftrightarrow(i\partial_t+\langle\nabla\rangle)\vec{u}=0,\\
(\Box+1)u=-f(u)\Longleftrightarrow(i\partial_t+\langle\nabla\rangle)\vec{u}=-\langle\nabla\rangle^{s_c-1}f(\langle\nabla\rangle^{-s_c}\Re\vec{u}),
\end{cases}
\end{align}

\begin{lemma}\label{long}  Let $I$ be a
time interval, $t_0\in I$ and $\vec{u},\vec{w}\in C(I;L^2(\R^d))$
satisfy
\begin{align*}
(i\partial_t+\langle\nabla\rangle)\vec{u}=&-\langle\nabla\rangle^{s_c-1}\big[f(u)+eq(u)\big]\\
(i\partial_t+\langle\nabla\rangle)\vec{w}=&-\langle\nabla\rangle^{s_c-1}\big[f(w)+eq(w)\big]
\end{align*}
for some function $eq(u),eq(w)$. Assume that for some constants
$M,E>0$, we have
\begin{align}\label{eq2.20}
\big\|w\big\|_{ST(I)}\leq M,\\\label{equ2.201}
\big\|\vec{u}\big\|_{L_t^\infty L^2_x(I\times
\R^d)}+\big\|\vec{w}\big\|_{L_t^\infty L^2_x(I\times \R^d)}\leq E,
\end{align}
 Let $t_0\in I$, and let $(u(t_0),u_t(t_0))$ be
close to $(w(t_0),w_t(t_0))$ in the sense that
\begin{equation} \label{eq2.22}
\big\|\gamma_0\big\|_{ST(I)}\leq\epsilon,
\end{equation}
where
$\vec{\gamma}_0=e^{i\langle\nabla\rangle(t-t_0)}(\vec{u}-\vec{w})(t_0)$
and $0<\epsilon<\epsilon_1=\epsilon_1( M, E)$ is a small constant.
Assume also that we have smallness conditions
\begin{equation}\label{equ2.21}
\big\|(eq(u),eq(w))\big\|_{ST^*(I)}\leq\epsilon,
\end{equation}
where $\epsilon$ is as above and
$$ST^*(I)=[W]^\ast(I)\oplus L_t^1(I;B_{2,2}^{s_c-1}(\R^d)).$$ Then we
conclude that
\begin{equation}\label{eq2.23}
\begin{aligned}
\big\|u-w\big\|_{ST(I)}\leq & C(M,E)\epsilon,\\
\big\|u\big\|_{ST(I)}\leq & C(M,E).
\end{aligned}
\end{equation}
\end{lemma}
\begin{proof}
Since $\|w\|_{ST(I)}\leq M$, there exists a partition of the right
half of $I$ at $t_0$:
$$t_0<t_1<\cdots<t_N,~I_j=(t_j,t_{j+1}),~I\cap(t_0,\infty)=(t_0,t_N),$$
such that $N\leq C(L,\delta)$ and for any $j=0,1,\cdots,N-1,$ we
have
\begin{equation}\label{ome}
\|w\|_{ST(I_j)}\leq\delta\ll1.
\end{equation}
The estimate on the left half of $I$ at $t_0$ is analogue, we omit
it.

Let
\begin{equation}
\gamma(t)=u(t)-w(t),~\vec{\gamma}_j(t)=e^{i\langle\nabla\rangle(t-t_j)}\vec
\gamma(t_j),~0\leq j\leq N-1,
\end{equation}
then $\vec\gamma$ satisfies the following difference equation
\begin{align*}
\begin{cases}
(i\partial_t+\langle\nabla\rangle)\vec{\gamma}=-\langle\nabla\rangle^{s_c-1}\Big(\big[\gamma(\gamma^2+3\gamma\omega+3\omega^2)\big]+eq(u)-eq(w)\Big)\\
\vec{\gamma}(t_j)=\vec{\gamma}_j(t_j),
\end{cases}
\end{align*}
which implies that
\begin{align*}
\vec{\gamma}(t)=&\vec{\gamma}_j(t)+i\int_{t_j}^te^{i\langle\nabla\rangle(t-s)}\Big\{\langle\nabla\rangle^{s_c-1}
\Big(\big[\gamma(\gamma^2+3\gamma\omega+3\omega^2)\big]+eq(u)-eq(w)\Big)\Big\}ds,\\
\vec{\gamma}_{j+1}(t)=&\vec{\gamma}_j(t)+i\int_{t_j}^{t_{j+1}}e^{i\langle\nabla\rangle(t-s)}\Big\{\langle\nabla\rangle^{s_c-1}
\Big(\big[\gamma(\gamma^2+3\gamma\omega+3\omega^2)\big]+eq(u)-eq(w)\Big)\Big\}ds.
\end{align*}
By Strichartz estimate \eqref{str} and nonlinear estimate
\eqref{nle}, we have
\begin{align}\label{equ3}
&\|\gamma-\gamma_j\|_{ST(I_j)}+\|\gamma_{j+1}-\gamma_j\|_{ST(\R)}\\\nonumber
\lesssim&\big\|\gamma^3+3\gamma^2\omega+3\gamma\omega^2\big\|_{[W]^\ast(I_j)}+\big\|(eq(u),eq(w))\big\|_{ST^\ast(I_j)}\\\nonumber
\lesssim
&\|\gamma\|_{ST(I_j)}^{1+\frac{4}{d-1}}\|\gamma\|_{L_t^\infty\dot{H}^{s_c}}^\frac{2(d-3)}{d-1}
+\|\omega\|_{ST(I_j)}\|\gamma\|_{ST(I_j)}^{\frac{4}{d-1}}\|\gamma\|_{L_t^\infty\dot{H}^{s_c}}^\frac{2(d-3)}{d-1}\\\nonumber
&+\|\gamma\|_{ST(I_j)}\|\omega\|_{ST(I_j)}^{\frac{4}{d-1}}\|\omega\|_{L_t^\infty\dot{H}^{s_c}}^\frac{2(d-3)}{d-1}
+\big\|(eq(u),eq(w))\big\|_{ST^\ast(I_j)}.
\end{align}
Therefore, assuming that
\begin{equation}\label{laodong}
\|\gamma\|_{ST(I_j)}\leq\delta\ll1,~\forall~j=0,1,\cdots,N-1,
\end{equation}
then by \eqref{ome} and \eqref{equ3}, we have
\begin{equation}
\|\gamma\|_{ST(I_j)}+\|\gamma_{j+1}\|_{ST(t_{j+1},t_N)}\leq
C\|\gamma_j\|_{ST(t_j,t_N)}+\epsilon,
\end{equation}
for some absolute constant $C>0$. By \eqref{eq2.22} and iteration on
$j$, we obtain
\begin{equation}
\|\gamma\|_{ST(I)}\leq (2C)^N\epsilon\leq\frac{\delta}2,
\end{equation}
provided we choose $\epsilon_1$ sufficiently small. Hence the
assumption \eqref{laodong} is justified by continuity in $t$ and
induction on $j$. Then repeating the estimate \eqref{equ3} once
again, we can get the ST-norm estimate on $\gamma$, which implies
the Strichartz estimates on $u$.
\end{proof}

%%%%%%%%%%%%%%%%%%%%%%%%%%%%%%%%%%%%%%%%%%%%%%%%%%%%%%%%%%%%%%%%%%%%%%%%%%%%%%%%%%
\section{Profile decomposition}
\setcounter{section}{3}\setcounter{equation}{0}

In this section, we first recall the linear profile decomposition of
the sequence of $L_x^2$-bounded solutions of
$(i\pa_t+\langle\nabla\rangle)\vec{v}=0$ which was established in
\cite{IMN}. And then we show the nonlinear profile decomposition
which will be used to construct the critical element and obtain its
compactness properties in the next section.

 \subsection{Linear profile decomposition}
First, we give some notation. For any triple
$(t_n^j,x_n^j,h_n^j)\in\R\times\R^d\times(0,1]$ with arbitrary
suffix $n$ and $j$, let $\tau_n^j,~T_n^j$, and
$\langle\nabla\rangle_n^j$ respectively denote the scaled time
shift, the unitary and the self-adjoint operators in $L^2(\R^d)$,
defined by
\begin{equation}\label{equ6.1}
\tau_n^j=-\frac{t_n^j}{h_n^j},~T_n^j\varphi(x)=(h_n^j)^{-\frac{d}{2}}\varphi\Big(\frac{x-x_n^j}{h_n^j}\Big),~\langle\nabla\rangle_n^j=\sqrt{-\Delta
+(h_n^j)^2}.
\end{equation}

Now we can state the linear profile decomposition as follows
\begin{lemma}[Linear profile decomposition,
\cite{IMN}]\label{lem3.1} Let $\vec{v}_n(t)=e^{i\langle\nabla\rangle
t}\vec{v}_n(0)$ be a sequence of free Klein-Gordon solutions with
uniformly bounded $L^2_x(\R^d)$-norm. Then after replacing it with
some subsequence, there exist $K\in\{0,1,2\ldots,\infty\}$ and, for
each integer $j\in[0,K)$, $\varphi^j\in L^2(\mathbb{R}^d)$ and
$\{(t^j_n,
x^j_n,h_n^j)\}_{n\in\mathbb{N}}\subset\mathbb{R}\times\mathbb{R}^d\times(0,1]$
satisfying the following. Define $\vec{v}^j_n$ and
$\vec{\omega}^k_n$ for each $j<k\leq K$ by
\begin{equation}\label{equ3.1}
\vec{v}_n(t,x)=\sum\limits_{j=0}^{k-1}\vec{v}^j_n(t,x)+\vec{\omega}^k_n(t,x),\end{equation}
where \begin{equation}
\vec{v}^j_n(t,x)=e^{i\langle\nabla\rangle(t-t^j_n)}T_n^j\varphi^j(x)=T_n^j\Big(e^{i\langle\nabla\rangle_n^j\frac{t-t_n^j}{h_n^j}}\varphi^j\Big)
,
\end{equation}
then  we have
\begin{equation}\label{equ3.2}
\lim\limits_{k\rightarrow
K}\varlimsup\limits_{n\rightarrow\infty}\big\|\vec{\omega}^k_n\big\|_{L_t^\infty(\mathbb{R};B^{-\frac{d}{2}}_{\infty,\infty}(
\mathbb{R}^d))}=0,
\end{equation}
and for  any $l<j<k\leq K$ and any $t\in \mathbb{R}$,
\begin{align}\label{equ3.3}
\lim\limits_{n\rightarrow\infty}\big\langle\mu\vec{v}^l_n,
\mu\vec{v}^j_n\big\rangle_{L^2_x}^2=0=
\lim\limits_{n\rightarrow\infty}\big\langle\mu\vec{v}^j_n,
\mu\vec{\omega}^k_n\big\rangle_{L^2_x}^2,
\\\label{equ3.4}
\lim\limits_{n\rightarrow\infty}\bigg\{\Big|\frac{h_n^l}{h_n^j}\Big|+\Big|\frac{h_n^j}{h_n^l}\Big|+\frac{|t_n^j-t_n^k|+|x_n^j-x_n^k|}{h_n^l}\bigg\}
=+\infty,
\end{align}
where $\mu\in\mathcal{MC}$ and $\mathcal{MC}$ is defined to be
$$\mathcal{MC}=\Big\{\mu=\mathcal{F}^{-1}\tilde{\mu}\mathcal{F}|\ \tilde{\mu}\in
C(\mathbb{R}^d),\exists
\lim\limits_{|x|\rightarrow\infty}\tilde{\mu}(x)\in\mathbb{R}\Big\}.$$
Moreover, each sequence $\{h_n^j\}_{n\in\mathbb{N}}$ is either going
to $0$ or identically $1$ for all $n$.
\end{lemma}
\begin{remark}
We call $\{\vec{v}_n^j\}_{n\in\mathbb{N}}$ a free concentrating wave
for each $j$, and $\vec{w}_n^k$ the remainder. From \eqref{equ3.3},
we have the following asymptotic orthogonality
\begin{equation}\label{orth}
\lim_{k\to K}\lim\limits_{n\rightarrow+\infty}\Big(\|\mu\vec{v}_n(t)\|_{L^2}^2-\sum\limits_{j=0}^{k-1}\|\mu\vec{v}_n^j(t)\|_{L^2}^2
-\|\mu\vec{\omega}_n^k(t)\|_{L^2}^2\Big)=0.
\end{equation}
\end{remark}
%the following lemma is to ensure there exists only one profile and the profile is also in the set \mathcal{A}^+.

We remark the following estimates for $1<p<\infty,$
\begin{align}\label{equ1.6.1}
\Big\|\big[|\nabla|-\langle\nabla\rangle_n\big]\varphi\Big\|_p\lesssim&
h_n\big\|\langle\nabla/h_n\rangle^{-1}\varphi\big\|_p,
\end{align}
hold uniformly for $0<h_n\leq1,$ by Mihlin's theorem on Fourier
multipliers.

\subsection{Nonlinear profile decomposition}
After the linear profile decomposition of a sequence of initial data
in the last subsection, we now show the nonlinear profile
decomposition of a sequence of the solutions of \eqref{equ1} with
the same initial data in the space $H^{s_c}(\mathbb{R}^d)\times
H^{s_c-1}(\mathbb{R}^d).$

First we construct a nonlinear profile associated with a free
concentrating wave. Let $\vec{v}_n^j$ be a free concentrating wave
for a sequence $(t_n^j,x_n^j,h_n^j)\in\R\times\R^d\times(0,1]$,
\begin{align}
\begin{cases}
(i\partial_t+\langle\nabla\rangle)\vec{v}_n^j=0,\\
\vec{v}_n^j(t_n)=T_n\varphi^j(x),\ \varphi^j(x)\in L^2(\R^d).
\end{cases}
\end{align}
Then by Lemma \ref{lem3.1}, for a sequence of free Klein-Gordon
solutions $\{\vec{v}_n(t)=e^{it\langle\nabla\rangle}\vec{v}_n(0)\}$
with uniformly bounded $L_x^2(\R^d)$-norm, we have a sequence of the
free concentrating wave $\vec{v}_n^j(t,x)$ with
$\vec{v}_n^j(t_n^j)=T_n^j\varphi^j,~\varphi^j\in L^2(\R^d)$ for
$j=0,1,\cdots,k-1,$ such that
\begin{align*}
\vec{v}_n(t,x)=&\sum_{j=0}^{k-1}\vec{v}^j_n(t,x)+\vec{\omega}^k_n(t,x)\\
=&\sum_{j=0}^{k-1}e^{i\langle\nabla\rangle(t-t^j_n)}T_n^j\varphi^j(x)+\vec{\omega}^k_n(t,x)\\
=&\sum_{j=0}^{k-1}T_n^je^{i\big(\frac{t-t^j_n}{h_n^j}\big)\langle\nabla\rangle_n^j}\varphi^j+\vec{\omega}^k_n(t,x).
\end{align*}

Now for any free concentrating wave $\vec{v}_n^j$, we undo the group
action $T_n^j$ to look for the linear profile $\vec{V}^j_n$. Let
$$\vec{v}_n^j(t,x)=T_n^j\vec{V}_n^j\big((t-t_n^j)/h_n^j\big),$$
then we have
$$\vec{V}_n^j(t,x)=e^{it\langle\nabla\rangle_n^j}\varphi^j.$$

Next let $\vec{u}_n^j$ be the nonlinear solution with the same
initial data $\vec{v}_n^j(0)$
\begin{align}
\begin{cases}
\big(i\partial_t+\langle\nabla\rangle\big)\vec{u}_n^j=-\langle\nabla\rangle^{s_c-1}f\big(\Re\langle\nabla\rangle^{-s_c}\vec{u}_n^j\big),\\
\vec{u}_n^j(0)=\vec{v}_n^j(0)=T_n^j\vec{V}^j_n(\tau_n^j),
\end{cases}
\end{align}
where $\tau_n^j=-t_n^j/h_n^j.$ In order to look for the nonlinear
profile $\vec{U}_\infty^j$ associated with the free concentrating
wave $\vec{v}_n^j$, we also need undo the group action. Define
$$\vec{u}_n^j(t,x)=T_n^j\vec{U}_n^j\big((t-t_n^j)/h_n^j\big),$$
then $\vec{U}^j_n$ satisfies the rescaled equation
\begin{align*}\begin{cases}
\big(i\partial_t+\langle\nabla\rangle_n^j\big)\vec{U}_n^j=
-\big(\langle\nabla\rangle_n^j\big)^{s_c-1}f\big(\Re(\langle\nabla\rangle_n^j)^{-s_c}\vec{U}_n^j\big),\\
\vec{U}_n^j(\tau_n^j)=\vec{V}^j_n(\tau_n^j).
\end{cases}
\end{align*}
Extracting a subsequence, we may assume that there exist
$h_\infty^j\in\{0,1\}$ and $\tau_\infty^j\in[-\infty,+\infty]$ for
every $j$, such that as $n\to+\infty$
$$h_n^j\rightarrow h^j_\infty,~\text{and}~\tau_n^j\to\tau_\infty^j.$$
Thus we have the limit equations as follows
\begin{align*}\vec{V}_\infty^j=e^{it\langle\nabla\rangle_\infty^j}\varphi^j,\quad
\begin{cases}
\big(i\partial_t+\langle\nabla\rangle_\infty^j\big)\vec{U}_\infty^j=-\big(\langle\nabla\rangle_\infty^j\big)^{s_c-1}f\big(\hat{U}_\infty^j\big),\\
\vec{U}_\infty^j(\tau_\infty^j)=\vec{V}^j_\infty(\tau_\infty^j),
\end{cases}
\end{align*}
where $\hat{U}_\infty^j$ is denoted to be
\begin{align}\label{un1}
\hat{U}_\infty^j:=\Re\big(\langle\nabla\rangle_\infty^j\big)^{-s_c}\vec{U}_\infty^j=
\begin{cases}
\Re\langle\nabla\rangle^{-s_c}\vec{U}_\infty^j~~\text{if}~~h^j_\infty=1,\\
\Re|\nabla|^{-s_c}\vec{U}_\infty^j~~\text{if}~~h^j_\infty=0.
\end{cases}
\end{align}

We remark that the unique existence of a local solution
$\vec{U}_\infty^j$ around $t=\tau_\infty^j$ is known in all cases,
including $h_\infty^j=0$ and $\tau_\infty^j=\pm\infty$. We say that
$\vec{U}_\infty^j$ on the maximal existence interval is the
nonlinear profile corresponding to the free concentrating wave
$(\vec{v}_n^j;t_n^j,x_n^j,h_n^j)$.

The nonlinear concentrating wave $\vec{u}_{(n)}^j$ corresponding to
$\vec{v}_n^j$ is defined by
\begin{equation}\label{un}
\vec{u}_{(n)}^j(t,x):=T_n^j\vec{U}_\infty^j\big((t-t_n^j)/h_n^j\big).
\end{equation}
When $h_\infty^j=1$, $u_{(n)}^j$ solves \eqref{equ1}. While
$h_\infty^j=0,$ then it solves
\begin{align}\label{errsm}\begin{cases}
(i\partial_t+\langle\nabla\rangle)\vec{u}_{(n)}^j
=\big(\langle\nabla\rangle-|\nabla|\big)\vec{u}_{(n)}^j-|\nabla|^{s_c-1}f\big(|\nabla|^{-s_c}\langle\nabla\rangle^{s_c}
u_{(n)}^j\big),\\
\vec{u}_{(n)}^j(0)=T_n^j\vec{U}_\infty^j(\tau_n^j).
\end{cases}
\end{align}
The existence time interval of $u_{(n)}^j$ may be finite and even go
to $0$, however, we have
\begin{equation}\label{lea}
\begin{split}
\|\vec{u}_n^j(0)&-\vec{u}_{(n)}^j(0)\|_{L_x^2}=\big\|T_n^j\vec{V}^j_n(\tau_n^j)-T_n^j\vec{U}_\infty^j(\tau_n^j)\big\|_{L_x^2}\\
\leq&\big\|\vec{V}^j_n(\tau_n^j)-\vec{V}_\infty^j(\tau_n^j)\big\|_{L_x^2}+\big\|\vec{V}^j_\infty(\tau_n^j)-\vec{U}_\infty^j(\tau_n^j)\big\|_{L_x^2}
\to0,
\end{split}
\end{equation}
as $n\to+\infty.$

Let $u_n$ be a sequence of  solutions of \eqref{equ1} around $t=0$,
and let $v_n$ be the sequence of the free solutions with the same
initial data. By  Lemma \ref{lem3.1}, we have the linear profile
decomposition for $\{\vec{v}_n\}$ as follows
\begin{equation*}
\vec{v}_n=\sum\limits_{j=0}^{k-1}\vec{v}^j_n+\vec{\omega}^k_n,\quad
\vec{v}^j_n=e^{i\langle\nabla\rangle(t-t^j_n)}T_n^j\varphi^j.
\end{equation*}
Now we define the nonlinear profile decomposition as follows.
\begin{definition}
[Nonlinear profile decomposition] Let
$\{\vec{v}_n^j\}_{n\in\mathbb{N}}$ be the free concentrating wave,
and $\{\vec{u}_{(n)}^j\}_{n\in\mathbb{N}}$ be the sequence of the
nonlinear concentrating wave corresponding to
$\{\vec{v}_n^j\}_{n\in\mathbb{N}}$. Then we define the nonlinear
profile decomposition of $u_n$ by
\begin{equation}\label{nonlineard}
\vec{u}_{(n)}^{<k}:=\sum\limits_{j=0}^{k-1}\vec{u}_{(n)}^j=\sum_{j=0}^{k-1}T_n^j\vec{U}_\infty^j\big((t-t_n^j)/h_n^j\big).
\end{equation}
\end{definition}

We will show that $\vec{u}_{(n)}^{<k}+\vec{\omega}_n^k$ is a good
approximation for $\vec{u}_n$ provided that each nonlinear profile
has finite global Strichartz norm.

Next we define the Strichartz norms for the nonlinear profile
decomposition. Recall that $ST(I)$ and $ST^\ast(I)$ are the
functions spaces on $I\times\R^d$ defined as above
\begin{align*}
ST(I)=&[W](I)=L_t^{\frac{2(d+1)}{d-1}}(I;B^{s_c-\frac12}_{\frac{2(d+1)}{d-1},2}(\mathbb{R}^d)),\\
ST^\ast(I)=&[W]^*(I)\oplus L_t^1(I; B_{2,2}^{s_c-1}(\R^d)).
\end{align*}
And the Strichartz norm for the nonlinear profile $\hat{U}_\infty^j$
is defined by
\begin{align}
ST_\infty^j(I):= \begin{cases} ST(I)~~\qquad\qquad\text{if}~~h_\infty^j=1,\\
L_t^q(I;\dot{B}^\frac{d-3}2_{q,2})~~\big(q=\frac{2(d+1)}{d-1}\big)~~\text{if}~~h_\infty^j=0.
\end{cases}
\end{align}

The following two lemmas derive from Lemma \ref{lem3.1} and the
perturbation lemma. The first lemma concerns the orthogonality in
the Strichartz norms.
%the following lemma is prepared for the next lemma.
\begin{lemma}\label{lem3.3}
Assume that in \eqref{nonlineard}, we have
\begin{equation}\label{equ3.7}
\|\hat{U}^j_\infty\|_{ST_\infty^j(\mathbb{R})}+\|\vec{U}^j_\infty\|_{L_t^\infty
L^2_x(\mathbb{R})}<+\infty,\ \forall\ j<k.
\end{equation}
Then, for any finite interval $I, j<k,$ one has
\begin{align}\label{equ3.8}
\varlimsup\limits_{n\rightarrow\infty}\|u_{(n)}^j\|_{ST(I)}&\lesssim\|\hat{U}_\infty^j\|_{ST_\infty^j(\mathbb{R})},\\\label{equ3.9}
\varlimsup\limits_{n\rightarrow\infty}\|u_{(n)}^{<k}\|_{ST(I)}^2&\lesssim\varlimsup\limits_{n\rightarrow\infty}
\sum\limits_{j=0}^{k-1}\big\|u_{(n)}^j\big\|_{ST(\mathbb{R})}^2,
\end{align}
where the implicit constants is independent of $I$ and $j$.
Furthermore, we have
\begin{equation}\label{equ3.10}
\lim\limits_{n\rightarrow\infty}\bigg\|f\big(u_{(n)}^{<k}\big)-\sum\limits_{j=0}^{k-1}
\Big(\frac{\langle\nabla\rangle_\infty^j}{\langle\nabla\rangle}\Big)^{s_c-1}
f\Big(\Big(\frac{\langle\nabla\rangle}{\langle\nabla\rangle_\infty^j}\Big)^{s_c}
u_{(n)}^j\Big)\bigg\|_{ST^*(I)}=0,
\end{equation}
where $f(u)=|u|^2u.$
\end{lemma}

\begin{proof} {\bf Proof of \eqref{equ3.8}: Case 1: $h_\infty^j=1.$}

It is easy to see that  $u_{(n)}^ j$ is just a sequence of
space-time translations of $\hat{U}_\infty^j$ in this case. And so
\eqref{equ3.8} follows in this case.

{\bf Case 2: $h_\infty^j=0$.}

We drop the superscript $j$ in the following. Using the definition
of $u_{(n)}$ and $\hat{U}_\infty$, we derive
$$u_{(n)}(t,x)=h_n^{s_c}T_n|\nabla|^{s_c}\langle\nabla\rangle_n^{-s_c}\widehat{U}_\infty\big((t-t_n)/h_n\big).$$
By Sobolev embedding $\dot{B}^0_{p,2}\subset L^p$ with $p\geq2$ in
the lower frequencies and scaling, one has for
$s:=s_c-\frac12=\frac{d-3}2,\ p=\frac{2(d+1)}{d-1},$
\begin{align*}
\big\|u_{(n)}\big\|_{B^s_{p,2}}\simeq&\big\|u_{(n)}\big\|_{L^p}+\big\|2^{js}
\|\Delta_j u_{(n)}\|_{L^p} \big\|_{l^2_{j\in \mathbb{N}}}\\
\lesssim&\big\| \|\Delta_j u_{(n)}\|_{L^p}
\big\|_{l^2_{j\in\mathbb{Z}^-}}+\big\|2^{js}
\|\Delta_j u_{(n)}\|_{L^p} \big\|_{l^2_{j\in \mathbb{N}}}\\
\lesssim&\Big\|2^{js}\big\|\Delta_j|\nabla|^{-s}\langle\nabla\rangle^{s}u_{(n)}\big\|_{L^p}\Big\|_
{l^2_{j\in\mathbb{Z}}}\\
\lesssim&h_n^{s_c-s-\frac{d}2+\frac{d}p}\Big\|
2^{js}\big\|\Delta_j|\nabla|^{s_c-s}
\langle\nabla\rangle_n^{s-s_c}\widehat{U}_\infty\big((t-t_n)/h_n\big)
\big\|_{L^p}\Big\|_
{l^2_{j\in\mathbb{Z}}}\\
\lesssim&h_n^{s_c-s-\frac{d}2+\frac{d}p}\Big\|
2^{js}\big\|\Delta_j\widehat{U}_\infty\big((t-t_n)/h_n\big) \big
\|_{L^p}\Big\|_
{l^2_{j\in\mathbb{Z}}}\\
\lesssim&h_n^{s_c-s-\frac{d}2+\frac{d}p}\Big\|\widehat{U}_\infty\big((t-t_n)/h_n\big)\Big\|_{\dot{B}^s_{p,2}}.
\end{align*}
Therefore, we obtain by scaling
\begin{align*}
\big\|u_{(n)}\big\|_{[W](I)}\lesssim&h_n^{s_c-s-\frac{d}2+\frac{d}p}\Big\|\big\|\widehat{U}_\infty\big((t-t_n)/h_n\big)\big\|_{\dot{B}^s_{p,2}}\Big\|_{
L_t^p(\R)}\\
\lesssim&\big\|\hat{U}_\infty\big\|_{L_t^p(\R;
\dot{B}^s_{p,2})}=\big\|\hat{U}_\infty\big\|_{ST_\infty(\R)},
\end{align*}
which concludes the proof of \eqref{equ3.8}.

{\bf Proof of \eqref{equ3.9}:} We estimate the left hand side of
\eqref{equ3.9} by
\begin{align*}
\big\|u_{(n)}^{<k}\big\|_{ST(I)}^2=&\bigg\|\sum_{j<k:
h_\infty^j=1}u_{(n)}^j+\sum_{j<k:
h_\infty^j=0}u_{(n)}^j\bigg\|_{ST(I)}^2\\
\lesssim&\Big\|\sum_{j<k:
h_\infty^j=1}u_{(n)}^j\Big\|_{ST(I)}^2+\Big\|\sum_{j<k:
h_\infty^j=0}u_{(n)}^j\Big\|_{ST(I)}^2.
\end{align*}
For the case $h_\infty^j=1$. Define $\widehat{U}_{\infty,R}^j,\
u_{(n),R}^j$ and $u_{(n),R}^{<k}$ by
\begin{align*}
\widehat{U}_{\infty,R}^j=\chi_R\widehat{U}_\infty^j,~u_{(n),R}^j=T_n^j\widehat{U}_{\infty,R}^j,~u_{(n),R}^{<k}:=\sum\limits_{j<k}u_{(n),R}^j,
\end{align*}
where $\chi_R(t,x)=\chi(t/R,x/R)$ and $\chi(t,x)\in
C_c^\infty(\R^{1+d})$ is the cut-off defined by
\begin{align*}
\chi(t,x)=\begin{cases}
1,\qquad |(t,x)|\leq1,\\
0,\qquad |(t,x)|\geq2.
\end{cases}
\end{align*}
Then we have
$$\Big\|\sum_{j<k:
h_\infty^j=1}u_{(n)}^j\Big\|_{ST(I)}^2\lesssim\Big\|\sum_{j<k:
h_\infty^j=1}u_{(n),R}^j\Big\|_{ST(I)}^2+\Big\|\sum_{j<k:
h_\infty^j=1}u_{(n)}^j-\sum_{j<k:
h_\infty^j=1}u_{(n),R}^j\Big\|_{ST(I)}^2.$$ On one hand, we know
that
$$\Big\|\sum_{j<k:
h_\infty^j=1}u_{(n)}^j-\sum_{j<k:
h_\infty^j=1}u_{(n),R}^j\Big\|_{ST(I)}\leq\sum\limits_{j<k:
h_\infty^j=1}\Big\|(1-\chi_R)\widehat{U}_\infty^j\Big\|_{ST(\R)}\rightarrow~0,$$
as $R \rightarrow~+\infty.$ On the other hand, by \eqref{equ3.4},
the similar orthogonality and approximation analysis as in
\cite{IMN}, we obtain
$$\varlimsup_{n\to\infty}\Big\|\sum_{j<k:
h_\infty^j=1}u_{(n)}^j\Big\|_{ST(I)}^2\lesssim\varlimsup_{n\to\infty}\Big\|\sum_{j<k:
h_\infty^j=1}u_{(n)}^j\Big\|_{ST(I)}^2\lesssim\varlimsup_{n\to\infty}\sum_{j<k:
h_\infty^j=1}\Big\|u_{(n)}^j\Big\|_{ST(I)}^2,$$ and for the case
$h_\infty^j=0$
$$\varlimsup_{n\to\infty}\Big\|\sum_{j<k:
h_\infty^j=0}u_{(n)}^j\Big\|_{ST(I)}^2\lesssim\varlimsup_{n\to\infty}\sum_{j<k:
h_\infty^j=0}\Big\|u_{(n)}^j\Big\|_{ST(I)}^2.$$

{\bf Proof of \eqref{equ3.10}:} By the definition of $u_{(n)}^j$ and
$\hat{U}_\infty^j$, we know that
$$u_{(n)}^j(x,t)=\Re\langle\nabla\rangle^{-s_c}\vec{u}_{(n)}^j(t,x)=\Re\langle\nabla\rangle^{-s_c}T_n^j\vec{U}_\infty^j\Big(\frac{t-t_n^j}{h_n^j}\Big)
=(h_n^j)^{s_c}T_n^j\Big(\frac{\langle\nabla\rangle_\infty^j}{\langle\nabla\rangle_n^j}\Big)^{s_c}
\hat{U}_\infty^j\Big(\frac{t-t_n^j}{h_n^j}\Big).$$ Let $u_{\langle
n\rangle}^{<k}(t,x)=\sum\limits_{j<k}u_{\langle n\rangle}^{j}(x,t),$
where $u_{\langle n\rangle}^{j}(x,t)$ is defined by
$$u_{\langle
n\rangle}^{j}(x,t)=\Big(\frac{\langle\nabla\rangle}{\langle\nabla\rangle_\infty^j}\Big)^{s_c}u_{(n)}^j
=(h_n^j)^{s_c}T_n^j
\hat{U}_\infty^j\Big(\frac{t-t_n^j}{h_n^j}\Big).$$ Then we have
\begin{align}\nonumber
&\bigg\|f\big(u_{(n)}^{<k}\big)-\sum\limits_{j=0}^{k-1}
\Big(\frac{\langle\nabla\rangle_\infty^j}{\langle\nabla\rangle}\Big)^{s_c-1}
f\Big(\Big(\frac{\langle\nabla\rangle}{\langle\nabla\rangle_\infty^j}\Big)^{s_c}
u_{(n)}^j\Big)\bigg\|_{ST^*(I)}\\\nonumber
\leq&\big\|f\big(u_{(n)}^{<k}\big)-f\big(u_{\langle
n\rangle}^{<k}\big)\big\|_{ST^*(I)}+\big\|f(u_{\langle
n\rangle}^{<k})-\sum\limits_{j<k}f(u_{\langle
n\rangle}^{j})\big\|_{ST^*(I)}\\\nonumber
&+\Big\|\sum\limits_{j<k}f(u_{\langle
n\rangle}^{j})-\sum\limits_{j<k}\Big(\frac{\langle\nabla\rangle_\infty^j}{\langle\nabla\rangle}\Big)^{s_c-1}f(u_{\langle
n\rangle}^{j})\Big\|_{ST^*(I)}\\\label{equ3.26}
\leq&\big\|f\big(u_{(n)}^{<k}\big)-f\big(u_{\langle
n\rangle}^{<k}\big)\big\|_{ST^*(I)}+\big\|f(u_{\langle
n\rangle}^{<k})-\sum\limits_{j<k}f(u_{\langle
n\rangle}^{j})\big\|_{ST^*(I)}\\\label{equ3.27}
&+\Big\|\sum\limits_{j<k:h_\infty^j=0}f(u_{\langle
n\rangle}^{j})-\sum\limits_{j<k:h_\infty^j=0}\Big(\frac{|\nabla|}{\langle\nabla\rangle}\Big)^{s_c-1}f(u_{\langle
n\rangle}^{j})\Big\|_{ST^*(I)}.
\end{align}
Using \eqref{equ3.4} and the approximation argument in \cite{IMN},
we get
$$\eqref{equ3.26}\to 0$$
as $n\to \infty.$ In addition, by $h_n^j\to 0$ as $n\to \infty,$ one
has
\begin{align*}
&\Big\|\sum\limits_{j<k:h_\infty^j=0}\Big(1-\Big(\frac{|\nabla|}{\langle\nabla\rangle}\Big)^{s_c-1}\Big)f\big(u_{\langle
n\rangle}^{j}\big)\Big\|_{ST^*(I)}\\
\lesssim&\sum\limits_{j<k:h_\infty^j=0}\Big\|\Big(1-\Big(\frac{|\nabla|}{\langle\nabla\rangle_n^j}\Big)^{s_c-1}\Big)
f\big(\hat{U}_\infty^{j}\big)\Big\|_{ST^*(I)}\to 0
\end{align*} as  $n\to \infty.$ Hence we obtain \eqref{equ3.10}. And so we
complete the proof of this lemma.
\end{proof}
With this preliminaries in hand, we now show that
$\vec{u}_{(n)}^{<k}+\vec{\omega}_n^k$ is a good approximation for
$\vec{u}_n$ provided that each nonlinear profile has finite global
Strichartz norm.
%the following lemma says that, not all profile of the critical element Strichartz norm is finite, otherwise it scatters.
\begin{lemma}\label{precldes}
Assume that $u_n$ is a sequence of local solutions of \eqref{equ1}
around $t=0$ obeying
$\varlimsup\limits_{n\rightarrow\infty}\|(u_n,\dot{u}_n)\|_{L_t^\infty(I_n;
H^{s_c}_x\times H^{s_c-1}_x)}<+\infty.$ Assume also that in its
nonlinear profile decomposition \eqref{nonlineard}, every nonlinear
profile $\vec{U}^j_\infty$ has finite global Strichartz and $L^2_x$
norms; that is
\begin{equation}\label{equ3.11}
\|\hat{U}^j_\infty\|_{ST_\infty^j(\mathbb{R})}+\|\vec{U}^j_\infty\|_{L_t^\infty
L_x^2(\mathbb{R})}<+\infty.
\end{equation}
Then $u_n$ is bounded for large $n$ in the Strichartz and the
$H^{s_c}$ norms, i.e.
\begin{equation}\label{equ3.12}
\varlimsup\limits_{n\rightarrow\infty}\big(\|u_n\|_{ST(\mathbb{R})}+\|\vec{u}_n\|_{L_t^\infty
L_x^2(\mathbb{R}\times \R^d)}\big)<+\infty.
\end{equation}

\end{lemma}

\begin{proof}
We only need to verify the conditions of Lemma \ref{long}. For this
purpose, by \eqref{errsm}, we derive that $u_{(n)}^{<k}+\omega_n^k$
satisfies that
\begin{align*}
(i\partial_{t}+\langle\nabla\rangle)\big(\vec{u}_{(n)}^{<k}+\vec{\omega}_n^k\big)
=-\langle\nabla\rangle^{s_c-1}\Big[f(u_{(n)}^{<k}+\omega_n^{k})+eq\big(u_{(n)}^{<k},\omega_n^{k}\big)\Big],
\end{align*}
where the error term $eq\big(u_{(n)}^{<k},\omega_n^{k}\big)$ is
\begin{align*}
eq\big(u_{(n)}^{<k},\omega_n^{k}\big)=
&\sum_{j<k}\langle\nabla\rangle^{1-s_c}\big(\langle\nabla\rangle-\langle\nabla\rangle_\infty^j\big)
\vec{u}_{(n)}^j+\Big[f(u_{(n)}^{<k}+\omega_n^{k})-f\big(u_{(n)}^{k}\big)\Big]
\\&+f(u_{(n)}^{<k})-\sum\limits_{j=0}^{k-1}
\Big(\frac{\langle\nabla\rangle_\infty^j}{\langle\nabla\rangle}\Big)^{s_c-1}
f\Big(\Big(\frac{\langle\nabla\rangle}{\langle\nabla\rangle_\infty^j}\Big)^{s_c}
u_{(n)}^j\Big).
\end{align*}

 First, by the definition of the nonlinear
concentrating wave $u_{(n)}^j$ and \eqref{lea}, we have
\begin{align}\nonumber
\Big\|\big(\vec{u}_{(n)}^{<k}(0)+\vec{w}_n^k(0)\big)-\vec{u}_n(0)\Big\|_{L^2_x}&\leq\sum\limits_{j=0}^{k-1}
\big\|\vec{u}_{(n)}^j(0)-\vec{u}_n^j(0)\big\|_{L^2_x}\rightarrow0,
\end{align}
as $n\rightarrow +\infty.$ This verifies the condition
\eqref{eq2.22} by the Strichartz estimate.

Next, by the linear profile decomposition in Lemma \ref{lem3.1}, we
get by \eqref{orth}
\begin{equation}\label{ee5}
\|\vec{u}_n(0)\|_{L^2}^2=\|\vec{v}_n(0)\|_{L^2}^2\geq\sum\limits_{j=0}^{k-1}\|\vec{v}_n^j(0)\|_{L^2}^2+o_n(1)
=\sum\limits_{j=0}^{k-1}\|\vec{u}_{(n)}^j(0)\|_{L^2}^2+o_n(1).
\end{equation}
Hence except for a finite set $J\subset\mathbb{N}$, the
$H^{s_c}_x\times H^{s_c-1}_x$-norm of
$\big(u_{(n)}^j(0),\dot{u}_{(n)}^j(0)\big)$ with $j\not\in J$ is
smaller than the iteration threshold (the small data scattering,
Theorem \ref{small}), and so
$$\|u_{(n)}^j\|_{ST(\mathbb{R})}\lesssim\|\vec{u}_{(n)}^j(0)\|_{L^2_x},\quad
j\not\in J.$$ This together with \eqref{equ3.8}, \eqref{equ3.9},
\eqref{equ3.11} and \eqref{ee5} yield that for any finite interval
$I$
\begin{align}\nonumber
\sup\limits_{k}\varlimsup\limits_{n\rightarrow\infty}\|u_{(n)}^{<k}\|_{ST(I)}^2&\lesssim\sum\limits_{j\in
J}\|u_{(n)}^{j}\|_{ST(I)}^2+\sum\limits_{j\not\in
J}\|u_{(n)}^j\|_{ST(\mathbb{R})}^2\\\label{ee6}
&\lesssim\sum\limits_{j\in
J}\|\hat{U}_\infty^j\|_{ST_\infty^j(\mathbb{R})}^2+\varlimsup\limits_{n\rightarrow\infty}\|\vec{u}_n(0)\|_{L^2}^2<+\infty.
\end{align}
This together with the Strichartz estimate for $\omega_n^k$ implies
that
$$\sup\limits_{k}\varlimsup\limits_{n\rightarrow\infty}\|u_{(n)}^{<k}+\omega_n^k\|_{ST(I)}<+\infty.$$
By the similar argument as above, we obtain for large $n$
$$\big\|\vec{u}_{(n)}^{<k}+\vec{w}_n^k\big\|_{L_t^\infty L_x^2}\leq
E_0.$$ Hence we verify the conditions \eqref{eq2.20} and
\eqref{equ2.201}.

It remains to verify the condition \eqref{equ2.21}. Using Lemma
\ref{lem3.1} and Lemma \ref{lem3.3}, we have
$$\big\|f(u_{(n)}^{<k}+\omega_n^k)-f(u_{(n)}^{<k})\big\|_{ST^*(I)}\rightarrow
0,$$ and
$$\bigg\|f\big(u_{(n)}^{<k}\big)-\sum\limits_{j=0}^{k-1}
\Big(\frac{\langle\nabla\rangle_\infty^j}{\langle\nabla\rangle}\Big)^{s_c-1}
f\Big(\Big(\frac{\langle\nabla\rangle}{\langle\nabla\rangle_\infty^j}\Big)^{s_c}
u_{(n)}^j\Big)\bigg\|_{ST^*(I)}\rightarrow 0,$$ as
$n\rightarrow+\infty.$ On the other hand, the linear part in
$eq\big(u_{(n)}^{<k},\omega_n^{k}\big)$ vanishes if $h_\infty^j=1$,
and from \eqref{equ1.6.1}, we know that it is controlled if
$h_\infty^j=0$ by
\begin{align}\nonumber
\Big\|\langle\nabla\rangle^{1-s_c}\big(\langle\nabla\rangle-|\nabla|\big)\vec{u}_{(n)}^j\Big\|_{L_t^1(I;
B^{s_c-1}_{2,2})}
\lesssim&|I|\cdot\Big\|\langle\nabla\rangle^{-1}\vec{u}_{(n)}^j\Big\|_{L_t^
\infty(\R; L_x^2)}\\\nonumber
\simeq&|I|\cdot\Big\|\langle\nabla/h_n^j\rangle^{-1}\vec{U}_{\infty}^j\Big\|_{L_t^
\infty(\R; L_x^2)}\\\label{xuxy0} \to&0,\quad\text{as}\quad
n\to+\infty,
\end{align}
by the continuity in $t$ and Lebesgue domainted convergence theorem for bounded $t$,
and by the scattering of $\hat{U}_\infty^j$ for $t\to\pm\infty,$
which follows from
$\big\|\hat{U}_\infty^j\big\|_{ST_\infty^j(\R)}<+\infty,$ and again
Lebesgue domainted convergence theorem.

Thus,
$\big\|eq\big(u_{(n)}^{<k},\omega_n^{k}\big)\big\|_{ST^\ast(I)}\to0,$
as $n\to+\infty.$

Therefore, for $k$ sufficiently close to $K$ and $n$ large enough,
the true solution $u_n$ and the near solution
$u_{(n)}^{<k}+\omega_n^k$ satisfy all the assumptions of the
perturbation   Lemma \ref{long}. Thus, we conclude this Theorem.
\end{proof}

\section{Concentration Compactness}
\setcounter{section}{4}\setcounter{equation}{0}

Using the profile decomposition in the previous section and the
perturbation theory, we argue in this section that if the scattering
result does not hold, then there must exist a minimal solution with
some good compactness properties.

\begin{proposition}\label{prop4.1}
 Suppose that  $E_{c}<+\infty$. Then there exists a
global solution $u_c$ of \eqref{equ1} satisfying
\begin{equation}\label{critical1}
\sup_{t\in \R}\big\|(u_c, \dot{u}_c)\big\|_{H^{s_c}\times
H^{s_c-1}}=E_{c},~\text{and}~ \|u_c\|_{ST(\mathbb{R})}=+\infty.
\end{equation}
Moreover, there exists $x(t):\mathbb{R}\rightarrow\mathbb{R}^d$ such
that the set $K=\big\{(u_c,\dot{u}_c)(t,x-x(t))\ \big|\
t\in\mathbb{R}^+\big\}$ is precompact in $H^{s_c}(\R^d)\times
H^{s_c-1}(\R^d)$.

\end{proposition}
\begin{proof}
By the definition of $E_{c}$, we can choose a sequence of solutions
to \eqref{equ1}: $\{u_n(t):~I_n\times\R\to\R\}$ such that
\begin{equation}\label{defocusing} \sup_{t\in
I_n}\big\|(u_n,\dot{u}_n)\big\|_{H^{s_c}\times H^{s_c-1}}\rightarrow
E_{c},\ \text{and}\ \|u_n\|_{ST(I_n)}\rightarrow+\infty,\ \text{as}\
n\rightarrow+\infty.\end{equation}

By Lemma \ref{lem3.1}, we have
\begin{equation}
\begin{cases}
 e^{it\langle\nabla\rangle}\vec{u}_n(0)=\sum\limits_{j=0}^{k-1}
\vec{v}^j_n+\vec{w}_n^k,\
\vec{v}^j_n=e^{i\langle\nabla\rangle(t-t^j_n)}T_n^j\varphi^j(x),\\
u_{(n)}^{<k}=\sum\limits_{j=0}^{k-1}u_{(n)}^j,\
\vec{u}^j_{(n)}(t,x)=T_n^j\vec{U}_\infty^j\big((t-t_n^j)/h_n^j\big),\\
\|\vec{v}_n^j(0)-\vec{u}_{(n)}^j(0)\|_{L^2_x}\rightarrow 0, \ as\
n\rightarrow+\infty.
\end{cases}
\end{equation}
Observing that
\begin{enumerate}
\item it follows from the definition of $E_{c}$ that every solution of \eqref{equ1} with $L_t^\infty(I;
H^{s_c}_x\times H^{s_c-1}_x)$-norm less than $E_{c}$ has global
finite Strichartz norm.
\item Lemma \ref{precldes} precludes that all the nonlinear
profiles $\vec{U}_\infty^j$ have finite global Strichartz norm.
\end{enumerate}
and by \eqref{orth}, we derive that there is only one profile, i.e.
$K=1,$ and so for large $n$
\begin{equation}
\sup_{t\in I}\big\|(u_{(n)}^0,\dot{u}_{(n)}^0)\big\|_{H^{s_c}\times
H^{s_c-1}}= E_{c},\ \|\hat{U}_\infty^0\|_{ST_\infty^0(I)}=+\infty,\
\lim\limits_{n\rightarrow+\infty}\|\vec{\omega}_n^1\|_{L_t^\infty
L^2_x}=0.
\end{equation}
If $h_n^0\to0,$ then $\hat{U}_\infty^0=\Re|\nabla|^{-s_c}\vec
U_\infty^0$ solves the $\dot H^{s_c}_x(\R^d)$-critical wave equation
$$\partial_{tt}u-\Delta u+|u|^2u=0$$
and satisfies
$$\sup_{t\in I}\big\|(\hat
U_\infty^0,\pa_t\hat U_\infty^0)\big\|_{H^{s_c}\times
H^{s_c-1}}=E_{c}<+\infty,~\big\|\hat
U_\infty^0\big\|_{L_t^q(I;\dot{B}^\frac{d-3}2_{q,2})}=+\infty,~q=\frac{2(d+1)}{d-1}.$$
But Bulut has shown that there is no such solution in
\cite{Blut2012,Blut2011}. Therefore, $h_n^0\equiv1.$ And so there
exist a sequence $(t_n,x_n)\in \mathbb{R}\times\mathbb{R}^d$ and
$\phi\in L^2(\mathbb{R}^d)$ such that along some subsequence,
\begin{equation}\label{sequa}
\big\|\vec{u}_n(0,x)-e^{-it_n\langle\nabla\rangle}\phi(x-x_n)\big\|_{L^2_x}\rightarrow
0, ~\text{as}~ n\rightarrow+\infty.\end{equation}

Now we show that $\hat
U_\infty^0=\Re\langle\nabla\rangle^{-s_c}\vec{U}_\infty^j$ is a
global solution. If not, then there exist a sequence $t_n\in \R$
which approaches the maximal existence time. Noting that $\big(\hat
U_\infty^0(t+t_n),\pa_t\hat U_\infty^0(t+t_n)\big)$ satisfies
\eqref{defocusing}, and then by the same argument as \eqref{sequa},
we deduce  that there exist another sequence $(t_n',x_n')\in
\mathbb{R}\times\mathbb{R}^d$ and for some $\psi\in L^2$ so that
\begin{equation}\label{sequa1}
\big\|\vec{U}_\infty^0(t_n)-e^{-it_n'\langle\nabla\rangle}\psi(x-x_n')\big\|_{L^2_x}\rightarrow
0, \end{equation} as $n\rightarrow\infty.$ We write
$\vec{v}:=e^{it\langle\nabla\rangle}\psi.$ From Strichartz estimate,
we know that for any $\varepsilon>0,$ there exist $\delta>0$ with
$I=[-\delta,\delta]$ so that
$$\big\|\langle\nabla\rangle^{-s_c}\vec{v}(t-t_n')\big\|_{ST(I)}\leq\frac{\eta_0}2,$$
where $\eta_0=\eta(d)$ is the threshold from the small data theory.
This together with \eqref{sequa1} shows that for sufficiently large
$n$
$$\big\|\langle\nabla\rangle^{-s_c}e^{it\langle\nabla\rangle}\vec
U_\infty^0(t_n)\big\|_{ST(I)}\leq2\eta_0.$$ Hence, by the small data
theory (Theorem \ref{small}), we derive that the solution $\vec
U_\infty^0$ exists on $[t_n-\delta,t_n+\delta]$ for large $n$, which
contradicts with the choice of $t_n$. Thus $\hat U_\infty^0$ is a
global solution and it is just the desired critical element $u_c$.
Moreover, since \eqref{equ1} is symmetric in $t$, we may assume that
\begin{equation}\label{forward}
\|u_c\|_{ST(0,+\infty)}=+\infty.
\end{equation} We call such $u$ a forward critical element.

Next we prove the precompactness of $K$. It is equivalent to show
the precompactness of $\{\vec{u}(t_n)\}$ in $L^2_x$ for any
$t_1,t_2,\cdots >0.$ It is easy to prove this by the continuity in
$t$ when  $t_n$ converges. Thus, we can suppose that
$t_n\rightarrow+\infty$. Applying the property of \eqref{sequa} to
the sequence of solution $\vec u(t+t_n)$, we get another sequence
$(t_n^\prime,x_n^\prime)\in \mathbb{R}^d$ and $\phi\in L^2$ such
that
\begin{equation}\label{equ7.2}
\big\|\vec{u}(t_n,x)-e^{-it_n^\prime\langle\nabla\rangle}\phi(x-x_n^\prime)\big\|_{L^2_x}\rightarrow
0,\ n\rightarrow\infty.
\end{equation}
If $t_n^\prime\rightarrow-\infty$, then we obtain by triangle
inequality
\begin{align*}
\|\langle\nabla\rangle^{-s_c}e^{it\langle\nabla\rangle}\vec{u}(t_n)\|_{ST(0,\infty)}\leq&
\|\langle\nabla\rangle^{-s_c}e^{it\langle\nabla\rangle}\big(\vec{u}(t_n)-e^{-it_n^\prime\langle\nabla\rangle}
\phi(x-x_n^\prime)\big)\|_{ST(0,\infty)}\\
&\ +\|\langle\nabla\rangle^{-s_c}e^{i(t-t_n^\prime)\langle\nabla\rangle}\phi(x-x_n^\prime)\|_{ST(0,\infty)}\\
\lesssim&\|\vec{u}(t_n,x)-e^{-it_n^\prime\langle\nabla\rangle}\phi(x-x_n^\prime)\|_{L^2_x}+
\|\langle\nabla\rangle^{-s_c}e^{it\langle\nabla\rangle}\phi\|_{ST(-t_n^\prime,\infty)}\\
\rightarrow&\ 0.
\end{align*}
Thus, by the small data theory, we can solve $u$ for $t<t_n$ with
large $n$ globally, which contradicts with its forward criticality.

If $t_n^\prime\rightarrow+\infty$, then one has
$$\big\|\langle\nabla\rangle^{-s_c}
e^{it\langle\nabla\rangle}\vec{u}(t_n)\big\|_{ST(-\infty,0)}=\big\|\langle\nabla\rangle^{-s_c}
e^{it\langle\nabla\rangle}\phi\big\|_{ST(-\infty,-t_n^\prime)}+o(1)\rightarrow
0.$$ Hence, we can solve $u$ for $t<t_n$ with large $n$ with
diminishing Strichartz norms. We give a contradiction since $u=0$ by
taking the limit.

Thus, $t_n'$ is bounded, which shows that $\{t_n^\prime\}$ is
precompact, so is $\vec{u}(t_n,x+x_n^\prime)$ in $L^2_x$ by
\eqref{equ7.2}.
\end{proof}

As a direct consequence of the above proposition, we have
\begin{corollary}\label{quick} (Compactness)
Let $u$ be a forward critical element. Then, for any $\eta>0$, there
exist $x(t):\R^+\to\R^d$ and $C(\eta)>0$ such that
\begin{equation}\label{comp} \sup_{t\in\R^+}\int_{|x-x(t)|\geq
C(\eta)}\Big(\big|\langle\nabla\rangle^{s_c}
u\big|^2+\big|\langle\nabla\rangle^{s_c-1}\dot{u}\big|^2\Big)dx\leq\eta.
\end{equation}
We refer to the function  $x(t)$ as the spatial center function, and
to $C(\eta)$ as the compactness modules function.
\end{corollary}

We remark that the small data theory shows that the
$H^{s_c}_x(\R^d)\times H^{s_c-1}_x(\R^d)$ norm of a blowup solution
must remain bounded from below. The fact that this norm is nonlocal
in odd space dimensions  reduces the efficacy of this statement. Our
next lemma gives a lower bound in a more suitable norm, and also a
mild control of $x(t)$.

\begin{lemma}\label{mild}
Let $u$ be a nonlinear strong solution of \eqref{equ1} as in
Proposition \ref{prop4.1}. Then
\begin{enumerate}
\item $((\nabla_{t,x}u,u)$ nontrivially$)$\quad We have
\begin{equation}\label{nont}
\inf\limits_{t\in\R^+}\int_{\R^d}\Big(|u|^\frac{d}2+|\nabla_{t,x}
u|^\frac{d}2\Big)dx\gtrsim1.\end{equation}
\item $($Control of $x(t))$ For some
large constant $C_u$, we have for any $t_1, t_2\in\R^+$
\begin{equation}\label{xtc}
|x(t_1)-x(t_2)|\leq|t_1-t_2|+2C_u.
\end{equation}
\end{enumerate}
\end{lemma}

\begin{proof} The proof is similar to \cite{KV}. But we  give a
sketch for the sake of completeness.

(1) It follows from the small data theory that
\begin{equation}\label{smlow}
\inf_{t\in\R^+}\big\|(u,u_t)\big\|_{H^{s_c}\times
H^{s_c-1}}\gtrsim1,
\end{equation}
otherwise $u$ would have finite spacetime norm which contradicts
with \eqref{critical1}.

On the other hand, it is easy to see that
$$\frac{\|f\|_{L_x^\frac{d}2(\R^d)}}{\|f\|_{H^{s_c-1}_x(\R^d)}}>0,$$
for any nonzero $\R^{1+d}$-valued $f\in H^{s_c-1}_x(\R^d)$. We note
that
\begin{equation*}
\begin{cases}
(i)\ \text{This ratio achieves a nonzero minimum on any compact set}\\
\qquad\text{that does not contain the zero function;}\\
(ii)\ \text{this ratio is invariant under translation;}\\
(iii)\ f:=\big(u_t,\langle\nabla\rangle u\big)~\text{and the set}\ K
\ \text{is precompact in}\ H^{s_c}_x(\R^d)\times H^{s_c-1}_x(\R^d).
\end{cases}
\end{equation*}
Combining these facts with \eqref{smlow}, we obtain that this ratio
is bounded from below, and so \eqref{nont} follows.

(2) Choose $\eta>0$ to be a small constant below the
$H^{s_c}_x(\R^d)\times H^{s_c-1}_x(\R^d)$ threshold for the small
data theory. By Corollary \ref{quick}, there is a constant
$C(\eta)>0$ such that
\begin{equation}\label{xcz}
\Big\|\phi\Big(\frac{x-x(t_1)}{C(\eta)}\Big)u(t_1,x)\Big\|_{{H}^{s_c}_x(\R^d)}+
\Big\|\phi\Big(\frac{x-x(t_1)}{C(\eta)}\Big)u_t(t_1,x)\Big\|_{{H}^{s_c-1}_x(\R^d)}\leq\eta,
\end{equation}
where \ $\phi:\R^d\to[0,+\infty),$
\begin{equation}
\phi(x)=\begin{cases} 1,\quad |x|\geq1,\\
0,\quad |x|\leq\frac12.
\end{cases}
\end{equation}
Hence, by the small data theory, there is a global solution to
\eqref{equ1} whose Cauchy data at time $t_1$ match the combination
of $\phi$ and $u$ given in \eqref{xcz}. Moreover, from the small
data theory, each critical Strichartz norm of this solution is
controlled by a multiple of $\eta$. It follows from domain of
dependence arguments that this new solution agrees with the original
$u$ on the set
$$\Omega(t):=\big\{x:~|x-x(t_1)|\geq|t-t_1|+C(\eta)\big\},\quad
t\in\R$$ and so by Sobolev embedding,
$$\big\|(u,\nabla_{t,x}u)\big\|_{L_x^{\frac{d}{2}}(\Omega(t))}\lesssim\eta,~\forall~t\in\R.$$
In particular, taking $t=t_2$, we get
\begin{equation}\label{t1j}
\int_{|x-x(t_1)|\geq|t_2-t_1|+C(\eta)}\Big(|u(t_2,x)|^\frac{d}2+|\nabla_{t,x}u(t_2,x)|^{\frac{d}{2}}\Big)dx\lesssim\eta.
\end{equation}
On the other hand, we have by Corollary \ref{quick} and Sobolev
embedding,
\begin{equation}
\int_{|x-x(t_2)|\geq
C(\eta)}\Big(|u(t_2,x)|^\frac{d}2+|\nabla_{t,x}u(t_2,x)|^{\frac{d}{2}}\Big)dx\lesssim\eta,
\end{equation}
This together with \eqref{nont} and \eqref{t1j} yield that
$$\big\{x:~|x-x(t_1)|\leq|t_2-t_1|+C(\eta)\big\}\cap\big\{x:~|x-x(t_2)|\leq
C(\eta)\big\}\neq\emptyset.$$ This concludes the proof of
\eqref{xtc}.
\end{proof}

The next corollary shows that the potential energy of the critical
element must concentrate.
\begin{corollary}[Concentration of potential
energy]\label{cor4.2} Let $u$ be a nonlinear strong solution of
\eqref{equ1} such that the set $K$ defined in Proposition
\ref{prop4.1} is precompact in $H^{s_c}(\R^d)\times
H^{s_c-1}(\R^d)$, and $E(u,\dot{u})\neq 0.$ For every $\tau>0$,
there exists two positive numbers $\alpha(\tau,u)$ and
$\beta(\tau,u)$ such that, for all time $t,$ there holds that
\begin{align}\label{big}
\alpha\leq\int_t^{t+\tau}\int_{\mathbb{R}^d}|u(s,x)|^{4}dxds
\leq\beta,
\end{align}
Moreover, combining this with Corollary \ref{quick} and Sobolev
embedding, we have for large $C=C(u)$ and all $t$
\begin{align}\label{big1}
\int_t^{t+1}\int_{|x-x(t)|\leq C}|u(s,x)|^{4}dxds\gtrsim1.
\end{align}
\end{corollary}

\begin{proof}
The bound from above follows from Sobolev's inequality and
$\sup\limits_{t\in \R}\big\|(u_c, \dot{u}_c)\big\|_{H^{s_c}\times
H^{s_c-1}}=E_{c}<+\infty$. Suppose the bound from below is not true.
Then there exist $\tau>0$ and a sequence $t_k$ such that
\begin{equation}\label{equ4.12}
\int_0^\tau\int_{\R^d}\big|u(t+t_k,x-x(t_k))\big|^4dxdt=\int_{t_k}^{t_k+\tau}\int_{\mathbb{R}^d}|u(t,x)|^{4}dxdt<\frac1k.
\end{equation}
Using the precompactness of $K$, we can extract a subsequence and
assume that
$$\big(\tau_{x(t_k)}u(t_k),\tau_{x(t_k)}u_t(t_k)\big)~\rightarrow~(U_0,U_1)\quad\text{in}~H^{s_c}_x(\R^d)\times
H^{s_c-1}_x(\R^d).$$ Let $U:~I\times\R^d\to\R$ be the nonlinear
strong solution of \eqref{equ1} with initial data $(U_0,U_1)$ at
time $t=0$. Then, $E(U,\dot{U})=E(u,\dot{u})\neq0$. By wellposedness
and \eqref{equ4.12}, we get
$$\int_{[0,\tau]\cap I}\int_{\mathbb{R}^d}|U(t,x)|^{4}dxdt=0$$
Hence, we have $U(t)=0$ for all $t$ in $(0,\tau)\cap I$, hence $U_t(t)=0$
for all such $t$. Consequently, $E(u,\dot{u})=0$. This is a
contradiction.
\end{proof}

\section{Extinction of the critical element}
\setcounter{section}{5}\setcounter{equation}{0}
 In this section, we
prove that the critical solution constructed in Section 4 does not
exist, thus ensuring that $E_{c} = +\infty$. This implies Theorem
\ref{theorem}.

\begin{proposition}\label{infty} There are no solutions to
\eqref{equ1} in the sense of Proposition \ref{prop4.1}.
\end{proposition}

\begin{proof}
We argue by contradiction. Assume there exists a solution $u:
\R^+\times\R^d\to \R$  such that the set $K$ defined in Proposition
\ref{prop4.1} is precompact in $H^{s_c}(\R^d)\times
H^{s_c-1}(\R^d)$. We will show that this scenario is inconsistent
with the following Morawetz inequality \cite{Br85,MC,MoS72}
\begin{equation}\label{mw}
\int_\R\int_{\R^d}\frac{|u(t,x)|^4}{|x|}dxdt\lesssim E(u,u_t).
\end{equation}

On one hand, since the solution $u$ has finite energy, the
right-hand side in the Morawetz inequality is finite and so
\begin{equation}\label{upper}
\int_0^T\int_{\R^d}\frac{|u(t,x)|^4}{|x|}dxdt\lesssim
E(u,u_t)\lesssim1,
\end{equation}
for any $T>0.$ On the other hand, we have concentration of potential
energy by Corollary \ref{cor4.2}. That is, there exists $C=C(u)$
such that
$$\int_{t_0}^{t_0+1}\int_{|x-x(t)|\leq C}|u(t,x)|^4dxdt\gtrsim 1,$$
for any $t_0\in\R$. Translating space so that $x(0)=0$ and employing
finite speed of propagation in the sense \eqref{xtc}
$$|x(t)-x(0)|\leq |t|+2c_u,$$
we deduce that for $T\geq1,$
\begin{align*}
\int_0^T\int_{\R^d}\frac{|u(t,x)|^4}{|x|}dxdt\gtrsim&\int_0^{T}\int_{|x-x(t)|\leq
C}\frac{|u(t,x)|^4}{|x|}dxdt\\
\gtrsim&\int_0^T\frac{dt}{1+t}\\
\gtrsim&\log(1+T),
\end{align*}
which contradicts with \eqref{upper} by choosing $T$ sufficiently
large depending on $u$. Hence we complete the proof of this
proposition.
\end{proof}

\vskip0.5cm

\textbf{Acknowledgements} The authors  were  supported by the NSF of China under
grant No.11171033, 11231006.

\begin{center}

\end{center}
\end{document}